%% file: klineq2.tex
\documentclass[a4paper,10pt]{amsart}

\usepackage[applemac]{inputenc}
\usepackage{amsmath, amsthm, amssymb}
\usepackage[all]{xy}
\usepackage{marginnote}
\usepackage{color}

\theoremstyle{plain}
\newtheorem{theorem}{Theorem}[section]
\newtheorem*{theorem*}{Theorem}

\newtheorem{proposition}[theorem]{Proposition}
\newtheorem{corollary}[theorem]{Corollary}
\newtheorem{lemma}[theorem]{Lemma}

\theoremstyle{definition}

\theoremstyle{remark}
\newtheorem{remark}[theorem]{Remark}

\DeclareMathOperator{\tr}{tr}
\DeclareMathOperator{\rk}{rk}
\DeclareMathOperator{\Id}{Id}

\def\Xint#1{\mathchoice
{\XXint\displaystyle\textstyle{#1}}%
{\XXint\textstyle\scriptstyle{#1}}%
{\XXint\scriptstyle\scriptscriptstyle{#1}}%
{\XXint\scriptscriptstyle\scriptscriptstyle{#1}}%
\!\int}
\def\XXint#1#2#3{{\setbox0=\hbox{$#1{#2#3}{\int}$ }
\vcenter{\hbox{$#2#3$ }}\kern-.6\wd0}}

\def\dashint{\Xint-}

\begin{document}
\bibliographystyle{alpha}

\title{Segre forms and Kobayashi--L\"Ubke inequality} 

\author{Simone Diverio }
\address{Simone Diverio \\ CNRS, Institut de Mathématiques de Jussieu--Paris Rive Gauche, UMR7586, Sorbonne Universités, UPMC Univ Paris 06, Univ Paris Diderot, Sorbonne Paris Cité, F-75005, Paris, France.}
\curraddr{Laboratorio Fibonacci, UMI 3483, Centro di Ricerca Matematica Ennio de Giorgi, Collegio Puteano, Scuola Normale Superiore, Piazza dei Cavalieri 3, I-56100 PISA.}
\email{simone.diverio@imj-prg.fr}

\thanks{The author is partially supported by the ANR Project \lq\lq GRACK\rq\rq}
\date{\today}


\begin{abstract}
Starting from the description of Segre forms as direct images of (powers of) the first Chern form of the (anti)tautological line bundle on the projectivized bundle of a holomorphic hermitian vector bundle, we derive a version of the pointwise Kobayashi--Lübke inequality.
\end{abstract}

\maketitle

\input{introduction}
\input{section1}

\input{section2}
\input{section3}

\bibliography{bibliography}{}

\end{document}

%% file: introduction.tex
\section{Introduction}

Let $E\to X$ be a rank $r$ holomorphic vector bundle over a compact complex manifold $X$ of dimension $n$, and let $\pi\colon P(E)\to X$ be the projectivized bundle of lines of $E$. Next, let $\mathcal O_E(-1)\subset\pi^* E$ be the tautological line bundle on $P(E)$ and set $u=c_1(\mathcal O_E(1))\in H^2(P(E),\mathbb Z)$. Then, the cohomology algebra $H^\bullet(P(E),\mathbb Z)$ can be identified with the algebra $H^\bullet(X,\mathbb Z)[u]$ with the unique relation
\begin{equation}\label{defchern}
u^r+\pi^*c_1(E)\cdot u^{r-1}+\cdots+\pi^*c_r(E)=0.
\end{equation}
It is well known that this can even be used in order to define the Chern classes $c_k(E)\in H^{2k}(X,\mathbb Z)$ of $E$. Let $c_\bullet(E)=1+c_1(E)+\cdots+c_r(E)\in H^\bullet(X,\mathbb Z)$ be the total Chern class of $E$. This is an invertible element of the cohomology algebra of $X$ and its inverse is by definition the total Segre class $s_\bullet(E)=1+s_1(E)+\cdots+s_n(E)$ of $E$. One thus finds
$$
\begin{aligned}
& s_1(E)=-c_1(E),\\
& s_2(E)=c_1(E)^2-c_2(E),\\
& s_3(E)=-c_1(E)^3+2\,c_1(E)\cdot c_2(E)-c_3(E),\\
&\dots
\end{aligned}
$$
and so forth. It is not difficult to see, using (\ref{defchern}), that one can recover Segre classes directly from the projectivized bundle of $E$ by a push-forward formula in cohomology, namely $s_k(E)=\pi_*u^{r-1+k}$.

Now, looking at these classes in the real cohomology algebra, the Chern--Weil theory gives us a way to represent them as closed $2k$-forms, for instance once the holomorphic vector bundle $E$ is endowed with a hermitian metric $h$: if $D_{E,h}$ is the Chern connection of $(E,h)$, and $\Theta(E,h)=D_{E,h}^2\in C^\infty_{1,1}(X,\operatorname{End}(E))$ its curvature, then the corresponding Chern forms $c_k(E,h)$ are computed using formally the identity
$$
\det\left(\operatorname{Id}+\frac{it}{2\pi}\,\Theta(E,h)\right)=\sum_{k=0}^r c_k(E,h)\, t^k,
$$
so that, $c_k(E,h)=\tr\bigl(\Lambda^k\frac i{2\pi}\Theta(E,h)\bigr)$.
Correspondingly, we can consider the Segre forms $s_k(E,h)$ built up starting from the Chern forms, so that for instance
$$
s_1(E,h)=-c_1(E,h),\quad s_2(E,h)=c_1(E,h)\wedge c_1(E,h)-c_2(E,h),\quad\dots
$$
and so on. These are, of course, special representatives of the Segre classes.

Once $E$ is endowed with a hermitian metric, we can naturally equip $\mathcal O_E(-1)\subset\pi^* E$ with a hermitian metric, which we still call $h$. Computing the corresponding Chern curvature of $\mathcal O_E(1)$ thus gives a special representative $\Xi:=\frac i{2\pi}\Theta(\mathcal O_E(1),h^{-1})$ of $u$. Finally, since $\pi\colon P(E)\to X$ is a proper submersion, the direct image $\pi_*(\Xi^{r-1+k})$ is a smooth closed $2k$-form, which clearly represents $s_k(E)$ so that \textsl{a priori} it differs from $s_k(E,h)$ by an exact form. The first remark is that this exact form is indeed zero.

\begin{proposition}\label{segre}
For each $k=0,\dots,n$, the equality
$$
\pi_*(\Xi^{r-1+k})=s_k(E,h).
$$
holds, where $s_0(E,h)$ is the function on $X$ constantly equal to $1$.
\end{proposition}

An analogous proposition has been firstly proven in \cite[Proposition 6]{Mou04}. The reader can also find this statement, for $X$ a projective manifold, in \cite[Proposition 3.1]{Gul12}. By the pointwise nature of its proof, Proposition \ref{segre} is indeed valid for any complex manifold: we shall give our proof of this fact in the next section.

The innocent-sounding proposition above has in fact a certain number of interesting consequences. The first one we would like to mention, which has already been observed in \cite{Gul12}, is the following.

\begin{theorem*}[D. Guler \cite{Gul12}]\label{Guler}
Let $(E,h)$ be a Griffiths positive holomorphic hermitian vector bundle on a projective manifold $X$. Then, the signed Segre form $(-1)^k\,s_k(E,h)$ is a positive $(k,k)$-form for each $k=1,\dots,n$.
\end{theorem*}

For definitions and basic facts about Griffiths' positivity and positivity of forms we refer to the all-inclusive book \cite{Dem}. This result should be put in perspective with \cite{FL83}, where it is shown, among other things, that, given any ample vector bundle $E$, each Schur polynomial in the Chern classes of $E$ is positive whenever integrated over any subvariety of the right dimension. Since signed Segre classes are particular Schur polynomials in the Chern classes, the above theorem can be seen as a partial pointwise metric counterpart of the above-mentioned result of Fulton and Lazarsfeld. 

\smallskip

The second consequence we want to consider is a pointwise inequality \textsl{\`a la} Kobayashi--L\"ubke for Hermite--Einstein vector bundles. To this effect, let us fix some notations. Let $(E,h)\to (X,\omega)$ be Hermite--Einstein with respect to the K\"ahler metric $\omega$. We recall that this means that there exists a real number $\lambda_{E,[\omega]}$ (called the \emph{slope of $E$ with respect to $\omega$}), which is \textsl{a posteriori} uniquely determined by $c_1(E)$ and the cohomology class $[\omega]$ of $\omega$, such that the following identity of $\operatorname{End}(E)$-valued $(n,n)$-forms holds everywhere on $X$:
$$
\frac i{2\pi}\,\Theta(E,h)\wedge\frac{\omega^{n-1}}{(n-1)!}=\lambda_{E,[\omega]}\,\frac{\omega^n}{n!}\operatorname{Id}_E.
$$
To simplify notations, we shall suppress the subscript $[\omega]$ in $\lambda_{E,[\omega]}$, whenever the Kähler class has been fixed once and for all. Taking the trace of both sides with respect to $\operatorname{End}(E)$, one gets
\begin{equation}\label{lambda}
c_1(E,h)\wedge\omega^{n-1}=\lambda_E\,\frac rn\,\omega^n.
\end{equation}
By integrating over $X$, this proves the above assertion about the dependance of $\lambda_E$ only upon $c_1(E)$ and $[\omega]$ and shows the correlation between $\lambda_E$ and the slope $\mu_E$ in the setting of stable vector bundles:
$$
\lambda_E=\frac n{\int_X\omega^n}\,\underbrace{\frac{\int_X c_1(E,h)\wedge\omega^{n-1}}{r}}_{=\deg_\omega E/\rk E=:\mu_E}.
$$
Next, recall that the classical Kobayashi--Lübke inequality states:

\begin{theorem*}[Kobayashi--Lübke inequality]
If $E$ admits a Hermite--Einstein metric $h$ with respect to $\omega$, then
$$
\bigl((r-1)\,c_1(E,h)^2-2r\,c_2(E,h)\bigr)\wedge\omega^{n-2}\le 0
$$
at every point of $X$. Moreover, the equality holds if and only if
$$
\frac i{2\pi}\,\Theta(E,h)=\frac 1r\,c_1(E,h)\otimes\operatorname{Id}_E,
$$
that is, if and only if $E$ is projectively flat.
\end{theorem*}

Here comes our inequality, which is ---as we shall see--- a different incarnation of the classical Kobayashi--Lübke inequality, involving an extra term of the form $c_1(E,h)\wedge\omega^{n-1}$. 

\begin{theorem}\label{KL}
If $E$ admits a Hermite--Einstein metric $h$ with respect to $\omega$, then
\begin{equation}\label{divin}
s_2(E,h)\wedge\omega^{n-2}\le\lambda_E\,\frac{r+1}{2n}\,c_1(E,h)\wedge\omega^{n-1}
\end{equation}
at every point of $X$. Moreover, the equality holds if and only if
$$
\frac i{2\pi}\,\Theta(E,h)=\frac{\lambda_E}{n}\,\omega\otimes\operatorname{Id}_E,
$$
so that in particular $E$ is projectively flat and 
$$
c_1(E,h)=\lambda_E\,\frac rn\,\omega.
$$
\end{theorem}

\begin{remark}
Observe that in the case of equality one obtains in particular that the first Chern class of $E$ is a (real) multiple of a Kähler class. Thus, if $\lambda_E$ is different from zero, \textsl{i.e.} if $c_1(E)\cdot[\omega]^{n-1}\ne 0$, then either $\det E$ or $\det E^*$ is a positive line bundle and therefore $X$ is projective. In particular, over a \emph{non projective} compact Kähler manifold the inequality (\ref{divin}) is always strict somewhere, provided $c_1(E)\cdot[\omega]^{n-1}\ne 0$. Therefore, in this case, by integrating (\ref{divin}) over $X$, one always has a strict cohomological inequality
$$
s_2(E)\cdot[\omega]^{n-2}<\lambda_E\,\frac{r+1}{2n}\,c_1(E)\cdot[\omega]^{n-1}.
$$
\end{remark}

Perhaps the main contribution of this note is to have given a new proof of the Kobayashi--Lübke inequality, very different in spirit from those already in the literature. Our proof seems to us very natural and might, as well, hint some new insights. 

Observe that, thanks to (\ref{lambda}), we may restate the inequality (\ref{divin}) in the equivalent form
\begin{equation}\label{divin2}
s_2(E,h)\wedge\omega^{n-2}\le\lambda_E^2\,\frac{r(r+1)}{2n^2}\,\omega^{n}.
\end{equation}
In order to recover the classical Kobayashi--Lübke inequality from Theorem \ref{KL}, given $(E,h)\to(X,\omega)$ Hermite--Einstein, it suffices to do the following standard trick: plug in (\ref{divin}) the first two Chern forms of the Hermite--Einstein vector bundle $(E^*\otimes E,\bar h^{-1}\otimes h)$ which are given by
$$
c_1(E^*\otimes E,\bar h^{-1}\otimes h)=0,\quad c_2(E^*\otimes E,\bar h^{-1}\otimes h)=2r\,c_2(E,h)-(r-1)\,c_1(E,h)^2.
$$
As pointed out to us by an anonymous referee, it is also equally possible to derive formally (\ref{divin2}) from the classical Kobayashi--Lübke inequality (see the end of Section \ref{ineqproof} for the details).

Integrating over $X$ the inequality (\ref{divin2}) and taking into account the Kobayashi--Hitchin correspondence, one gets the following statement of a somewhat more algebraic flavor.

\begin{corollary}
Let $E\to X$ be a holomorphic vector bundle over a compact Kähler manifold $X$. Suppose that $E$ is (poly)stable with slope $\mu_E$ with respect some Kähler class $[\omega]$. Then, the following inequality holds in cohomology:
$$
\bigl(c_1(E)^2-c_2(E)\bigr)\cdot[\omega]^{n-2}\le \mu_E^2\,\frac{r(r+1)}{2[\omega]^n}.
$$
\end{corollary}

If $X$ is moreover projective algebraic and $[\omega]=c_1(A)$ is taken to be the first Chern class of an ample line bundle $A\to X$, then the same conclusion can be shown to hold, more generally, if $E$ is only supposed to be semi--stable with respect to $[\omega]$. Indeed, thanks to the classical work of Donaldson \cite{Don85} (see also \cite[Theorem (VI.10.13)]{Kob87}), in this case $E$ admits approximate Hermite--Einstein metrics and we thus get some error term in (\ref{divin2}) which disappears integrating and passing to the limit, exactly as in \cite[Theorem (IV.5.7)]{Kob87}.

\subsubsection*{Acknowledgments} Firstly, we would like to friendly thank F. Campana who encouraged us to write down this note some time ago. We are also indebted with S. Boucksom, J. Cao, P. Dingoyan, P. Gauduchon, E. Mistretta, and R. A. Wentworth for extremely valuable discussions and suggestions. Finally, we thank B. Claudon and A. Höring for pointing out to us respectively the references \cite{Gul12} and \cite{Miy91}. Last but not least, we thank a first anonymous referee who suggested us how to formally derive inequality (\ref{divin2}) from Kobayashi--Lübke and a second anonymous referee for several useful comments and suggestions.

%% file: section1.tex
\section{Segre forms as direct images}

Let $(E,h)\to X$ be a rank $r$ holomorphic hermitian vector bundle on a complex manifold $X$ of dimension $n$. Let $D_{E,h}$ be the Chern connection of $(E,h)$, and $\Theta(E,h)=D_{E,h}^2\in C^\infty_{1,1}(X,\operatorname{End}(E))$ its curvature. Let $\pi\colon P(E)\to X$ be the projectivized bundle of lines of $E$, and $\mathcal O_E(-1)$ the tautological line bundle on $P(E)$. The metric $h$ on $E$ naturally induces a metric on $\mathcal O_E(-1)$, being $\mathcal O_E(-1)$ a subbundle of $\pi^*E$. 

Given a point $x_0\in X$, we shall compute its Chern curvature at an arbitrary point $(x_0,[v_0])\in P(E)$ in the fiber over $x_0$. To do so, let us fix local holomorphic coordinates $(z_1,\dots,z_n)$ on $X$ centered at $x_0$ and a local normal frame $(e_1,\dots,e_r)$ for $E$ at $x_0$, such that $e_r(x_0)=v_0/||v_0||_h$. These choices give us local holomorphic coordinates $(z_1,\dots,z_n,\xi_1,\dots,\xi_{r-1})$ on $P(E)$ centered at $(x_0,[v_0])$. A local holomorphic nonvanishing section for $\mathcal O_E(-1)$ is thus given by
$$
(z,\xi)\mapsto\eta(z,\xi)=e_r(z)+\sum_{\lambda=1}^{r-1}\xi_\lambda\,e_\lambda(z),
$$
and the Chern curvature of $\mathcal O_E(-1)$ at the point $(x_0,[v_0])$ is computed by
$$
\bigl(-\partial\bar\partial\log||\eta||_h^2\bigr)|_{(0,0)}.
$$
Now, since $(e_\lambda)$ is a local normal frame, by definition we have
$$
\langle e_\lambda(z),e_\mu(z)\rangle_h=\delta_{\lambda\mu}-\sum_{j,k=1}^nc_{jk\lambda\mu}z_j\bar z_k+O(|z|^3),
$$
where
$$
\begin{aligned}
\Theta_{x_0}(E,h) &=\sum_{\lambda,\mu=1}^{r}\sum_{j,k=1}^nc_{jk\lambda\mu}\,dz_j\wedge d\bar z_k\otimes e_\lambda^*\otimes e_\mu \\
&=\sum_{\lambda,\mu=1}^{r}\Theta_{\mu\lambda}\otimes e_\lambda^*\otimes e_\mu
\end{aligned}
$$
is the Chern curvature of $(E,h)$ at $x_0$. Here, we evidently have set 
$$
\Theta_{\mu\lambda}=\sum_{j,k=1}^nc_{jk\lambda\mu}\,dz_j\wedge d\bar z_k.
$$
Therefore,
$$
||\eta||_h^2=1+|\xi|^2-\sum_{j,k=1}^nc_{jkrr}z_j\bar z_k+O((|z|+|\xi|)^3),
$$
and the Chern curvature we wanted to compute is given by
\begin{equation}\label{curv}
\Theta_{(x_0,[v_0])}(\mathcal O_E(-1),h)=
-\sum_{\lambda=1}^rd\xi_\lambda\wedge d\bar\xi_\lambda+\sum_{j,k=1}^nc_{jkrr}\,dz_j\wedge d\bar z_k.
\end{equation}
Next, we rewrite this formula in more intrinsic terms. For this, we shall exhibit a natural (smooth) decomposition of $T_{P(E)}$ in vertical and horizontal distributions, which depend on the hermitian structure of $E$. First of all observe that, tautologically, the restriction of $\frac i{2\pi}\Theta(\mathcal O_E(-1),h)$ to any fiber $\pi^{-1}(x)\simeq P(E_x)$, $x\in X$, gives minus the Fubini--Study metric $\omega^{FS}_{P(E_{x}),h|_{E_{x}}}$ of $P(E_{x})$, with respect to the metric $h|_{E_{x}}$. In particular, the hermitian form on the holomorphic tangent space $T_{P(E)}$ associated to $\frac i{2\pi}\Theta(\mathcal O_E(-1),h)$ is negative definite on the relative tangent bundle $T_{P(E)/X}:=\ker d\pi\subset T_{P(E)}$. Thus, the orthogonal complement $T_{P(E)/X}^\perp$ to $T_{P(E)/X}$ with respect to $\frac i{2\pi}\Theta(\mathcal O_E(-1),h)$ gives rise to a smooth distribution of complex dimension $n$ such that
$$
T_{P(E)}\simeq_{C^\infty}T_{P(E)/X}\oplus T_{P(E)/X}^\perp,
$$
and $d\pi|_{T_{P(E)/X}^\perp}\colon T_{P(E)/X}^\perp\to T_X$ is a smooth linear isomorphism at every point of $P(E)$.

Using this decomposition of $T_{P(E)}$ into vertical and horizontal distributions, we see that the second term in (\ref{curv}) only acts on horizontal vectors and can be therefore identified (\textsl{via} $d\pi$) with the $(1,1)$-form on $T_{X,x_0}$ given by
$$
\frac 1{||v_0||_h^2}\,\bigl\langle\Theta_{x_0}(E,h)\cdot v_0,v_0\bigr\rangle_h.
$$
Set
$$
\Xi:=\frac{i}{2\pi}\Theta(\mathcal O_E(1),h^{-1}).
$$ 
Summing up, $\Xi$ is the real $(1,1)$-form representing $c_1(\mathcal O_E(1))$ that can be rewritten, using the decomposition above, as
$$
P(E)\ni (x,[v])\mapsto \Xi(x,[v])=\omega^{FS}_{P(E_x),h|_{E_x}}([v])-\vartheta^E_h(x,[v]),
$$
where we have set
$$
\vartheta^E_h(x,[v]):=\frac i{2\pi}\,\frac {\bigl\langle\Theta_{x}(E,h)\cdot v,v\bigr\rangle_h}{||v||_h^2}.
$$
\begin{proposition}
For any integer $0\le k\le n$, the direct image $\pi_*\Xi^{r-1+k}$ equals the $k$-th Segre form $s_k(E,h)$ of $(E,h)$.
\end{proposition}

As mentioned in the introduction, by the pointwise nature of this statement, no hypotheses of algebricity (or kählerness) nor of compactness of $X$ are indeed needed.

\begin{proof}
This is a pointwise computation. We have to show that, for an arbitrary $x_0\in X$, the $(k,k)$-form given by the integration of $\Xi^{r-1+k}$ over the fiber $P(E_{x_0})$ coincides with $s_k(E,h)$ at $x_0$, in some suitable local coordinates.

For this, let us keep notations as above and compute
\begin{equation}\label{power}
\Xi^{r-1+k}(x_0,[v]) =\sum_{i=0}^{r-1+k}\binom{r-1+k}{i}(-1)^i\,\bigl(\vartheta^E_h(x_0,[v])\bigr)^i\wedge\bigl(\omega^{FS}_{P(E_{x_0}),h|_{E_{x_0}}}([v])\bigr)^{r-1+k-i}.
\end{equation}
For obvious degree reasons, among all terms in (\ref{power}), the only one which survives once an integration over the fibers is performed is of course the one corresponding to $i=k$. We are therefore led to consider the integral
\begin{multline*}
\int_{P(E_{x_0})}\Xi^{r-1+k} \\ =
(-1)^k\binom{r-1+k}{k}\int_{P(E_{x_0})}\bigl(\vartheta^E_h(x_0,[v])\bigr)^k\wedge\bigl(\omega^{FS}_{P(E_{x_0}),h|_{E_{x_0}}}([v])\bigr)^{r-1} \\
=(-1)^k(r-1)!\binom{r-1+k}{k} \times \\ \times\int_{P(E_{x_0})}\left(\frac i{2\pi}\,\frac {\bigl\langle\Theta_{x_0}(E,h)\cdot v,v\bigr\rangle_h}{||v||_h^2}\right)^kd\operatorname{Vol}^{FS}_{P(E_{x_0}),h|_{E_{x_0}}}([v]) \\
=(-1)^k(r-1)!\binom{r-1+k}{k} \times \\ \times \int_{P(E_{x_0})}\left(\frac i{2\pi}\,\frac {\sum_{\lambda,\mu=1}^r\Theta_{\mu\lambda}v_\lambda\bar v_\mu}{||v||_h^2}\right)^kd\operatorname{Vol}^{FS}_{P(E_{x_0}),h|_{E_{x_0}}}([v]).
\end{multline*}
Here, $d\operatorname{Vol}^{FS}_{P(E_{x_0}),h|_{E_{x_0}}}$ is the Fubini--Study volume element on $P(E_{x_0})$ and the $v_\lambda$'s, $\lambda=1,\dots,r$, are the coordinates of $v\in E_{x_0}$ with respect to the basis $(e_1(x_0),\dots,e_r(x_0))$. The $(1,1)$-forms $\Theta_{\mu\lambda}$ can be of course considered as constants, since the point $x_0$ is kept fixed. 

Now, given a finite dimensional hermitian vector space $V$, its projectivization $p\colon V\setminus\{0\}\to P(V)$, call $U(V)$ the unit sphere of $V$, and $d\operatorname{Vol}^{FS}$ and $d\sigma$ respectively the Fubini--Study volume element on $P(V)$ and the Lebesgue measure on $U(V)$ induced by the hermitian structure of $V$. The volume of $U(V)$ with respect to $d\sigma$ is thus given by $2\pi^{\dim V}/(\dim V-1)!$. Then, given an integrable function $f\colon P(V)\to\mathbb C$, it is well-known (see for instance \cite[Lemma 1, Section 13.4, Chapter 3]{Chi89}) that the following equality holds:
$$
\int_{P(V)} f\,d\operatorname{Vol}^{FS} = \frac {1}{2\pi^{\dim V}}\int_{U(V)}(f\circ p)\,d\sigma=\frac 1{(\dim V-1)!}\,\dashint_{U(V)}(f\circ p)\,d\sigma,
$$
where $\dashint$ stands for the average integral. Therefore, what remains to show is that
$$
(-1)^k{\binom{r-1+k}{k}}\dashint_{U(E_{x_0})}\biggl(\frac i{2\pi}\sum_{\lambda,\mu=1}^r\Theta_{\mu\lambda}v_\lambda\bar v_\mu\biggr)^k d\sigma(v)=s_k(E,h)(x_0).
$$
This will be achieved in the subsection below and the proposition is proved.
\end{proof}

\subsection{An elementary lemma}

Let $V$ be a complex vector space of complex dimension $r$ and $\langle\cdot,\cdot\rangle$ a hermitian inner product on $V$. Next, given a positive integer $k$, consider on the space $\mathcal H$ of self-adjoint linear operators on $V$ the following homogeneous function of degree $k$:
$$
\begin{aligned}
\phi_k\colon & \mathcal H\to\mathbb R \\
& T\mapsto\dashint_{S^{2r-1}}\langle T(v),v\rangle^k\,d\sigma(v):=\frac{(r-1)!}{2\pi^r}\int_{S^{2r-1}}\langle T(v),v\rangle^k\,d\sigma(v).
\end{aligned}
$$
Here, $S^{2r-1}$ is the unitary sphere with respect to the fixed hermitian inner product, $d\sigma$ is the Lebesgue measure, and $2\pi^r/(r-1)!$ the corresponding volume of the sphere.

Let us compute what this function gives, in terms of the eigenvalues of $T\in\mathcal H$. So, fix a unitary basis $\{e_1,\dots,e_r\}$ of $V$ which diagonalizes $T$. Suppose the matrix $\Theta$ of $T$ relative to this basis be given by $\Theta=\operatorname{diag}(\lambda_1,\dots,\lambda_r)$.
With this choices, we have
$$
\phi_k(T)=\sum_{j_1,\dots,j_k=1}^r\lambda_{j_1}\cdots\lambda_{j_k}\dashint_{S^{2r-1}}|z_{j_1}|^2\cdots|z_{j_k}|^2\,d\sigma(z),
$$
where the $z_j$'s are coordinates with respect to the $\{e_j\}$ basis. For $j_1,\dots,j_k=1,\dots,r$, let us call 
$$
I(j_1,\dots,j_k):=\dashint_{S^{2r-1}}|z_{j_1}|^2\cdots|z_{j_k}|^2\,d\sigma(z).
$$
\begin{lemma}
We have
$$
I(j_1,\dots,j_k)=\frac{m_1!\cdots m_r!(r-1)!}{(r-1+k)!},
$$
where $m_\ell$, $\ell=1,\dots r$, is the number of times that $\ell$ appears among $j_1,\dots,j_k$.
\end{lemma}

\begin{proof}
Consider the integral
$$
\int_{\mathbb C^r}e^{-|z|^2}|z_{j_1}|^2\cdots|z_{j_k}|^2\,dz.
$$
From the one hand, passing to polar coordinates, we have
$$
\begin{aligned}
\int_{\mathbb C^r}e^{-|z|^2}|z_{j_1}|^2\cdots|z_{j_k}|^2\,dz &= \int_{\mathbb C^r}e^{-|z|^2}|z|^{2k}\frac{|z_{j_1}|^2\cdots|z_{j_k}|^2}{|z|^{2k}}\,dz \\
&= \frac{2\pi^r}{(r-1)!}\,I(j_1,\dots,j_k)\int_0^{+\infty}e^{-\rho^2}\rho^{2k+2r-1}\,d\rho \\
&= \frac{2\pi^r}{(r-1)!}\,I(j_1,\dots,j_k)\,\frac{\Gamma(r+k)}2 \\
&= \frac{\pi^r(r-1+k)!}{(r-1)!}\,I(j_1,\dots,j_k),
\end{aligned}
$$
where $\Gamma$ is the Euler's gamma function.
On the other hand, by separating the variables and then passing to polar coordinates, we obtain
$$
\begin{aligned}
\int_{\mathbb C^r}e^{-|z|^2}|z_{j_1}|^2\cdots|z_{j_k}|^2\,dz &=
\int_{\mathbb C^r}e^{-|z|^2}|z_{1}|^{2m_1}\cdots|z_{r}|^{2m_r}\,dz \\
&=\prod_{j=1}^r\int_{\mathbb C}e^{-|z_j|^2}|z_j|^{2m_j}\,dz_j \\
&= \prod_{j=1}^r 2\pi\int_0^{+\infty}e^{-\rho^2}\rho^{2m_j+1}\,d\rho \\
&= (2\pi)^r \prod_{j=1}^r \frac{\Gamma(m_j+1)}2 = \pi^r\,m_1!\cdots m_r!.
\end{aligned}
$$
Putting these together, we obtain the desired result.
\end{proof}

Now, this lemma tells us that, for $T\in\mathcal H$, we have
$$
\begin{aligned}
\phi_k(T) &=\sum_{j_1,\dots,j_k=1}^r\frac{m_1!\cdots m_r!(r-1)!}{(r-1+k)!}\,\lambda_{j_1}\cdots\lambda_{j_k} \\
&= \frac{k!(r-1)!}{(r-1+k)!}\sum_{j_1,\dots,j_k=1}^r\frac{m_1!\cdots m_r!}{k!}\,\lambda_{j_1}\cdots\lambda_{j_k} \\
&= \frac 1{\binom{r-1+k}{k}}\,\sigma_k(\lambda_1,\dots,\lambda_r),
\end{aligned}
$$
where $\sigma_k$ is the complete homogeneous symmetric polynomial of degree $k$ in $r$ variables. Thus, it has a unique expression in terms of the elementary symmetric polynomials $\gamma_j$'s in the eigenvalues of $T$, which are nothing but the traces of the exterior powers of $T$. This can be explicitly obtained by the well-known relation
$$
\sum_{j=0}^m(-1)^j \sigma_j(\lambda_1,\dots,\lambda_r)\cdot \gamma_{m-j}(\lambda_1,\dots,\lambda_r)=0,
$$
which is valid for all integer $m>0$. Here are the first few as an example:
$$
\begin{aligned}
&\phi_1(T)=-\frac 1r\,\tr(T), \\
&\phi_2(T)=\frac 2{r(r+1)}\,\bigl((\tr(T)^2-\tr(\Lambda^2T)\bigr),\\
&\phi_3(T)=-\frac 6{r(r+1)(r+2)}\,\bigl((\tr(T)^3-2\tr(T)\tr(\Lambda^2T)+\tr(\Lambda^3T)\bigr),\\
&\dots
\end{aligned}
$$
Thus, if we formally compute $\phi_k\bigl(\frac i{2\pi}\Theta_{x_0}(E,h)\bigr)$, where $\Theta$ is the Chern curvature of a holomorphic hermitian vector bundle $(E,h)\to X$ at a given point $x_0\in X$, we exactly obtain
$$
\phi_k\biggl(\frac i{2\pi}\Theta_{x_0}(E,h)\biggr)=\frac{(-1)^k}{\binom{r-1+k}{k}}\,s_k(E,h)(x_0).
$$

%% file: section2.tex
\section{A Kobayashi--L\"ubke type inequality}\label{ineqproof}

Let $(E,h)\to (X,\omega)$ be a holomorphic hermitian rank $r$ vector bundle on a Kähler $n$-dimensional compact manifold, with Kähler metric $\omega$ and suppose that $(E,h)$ is Hermite--Einstein with respect to $\omega$. Then, there exists a real number $\lambda_E$, such that
\begin{equation}\label{HE}
\frac{i}{2\pi}\,\Theta(E,h)\wedge\frac{\omega^{n-1}}{(n-1)!}=\lambda_E\,\frac{\omega^n}{n!}\otimes\Id_E.
\end{equation}
Now, we shall explain how to describe the Hermite--Einstein property of $\Theta(E,h)$ in terms of the curvature of $(\mathcal O_E(1),h^{-1})$. Keeping notations as in the preceding section, we have the following proposition.

\begin{proposition}\label{hepe}
The hermitian vector bundle $(E,h)\to (X,\omega)$ is Hermite--Einstein if and only if the following identity holds:
\begin{equation}\label{HEPE}
\frac{\Xi^{r}}{r!}\wedge\frac{\pi^*\omega^{n-1}}{(n-1)!} = -\lambda_E\,\frac{\Xi^{r-1}}{(r-1)!}\wedge\frac{\pi^*\omega^{n}}{n!}.
\end{equation}
\end{proposition}

\begin{proof}
For the reader convenience, let us denote by $\omega^{FS}_x$ the Fubini--Study metric $\omega^{FS}_{P(E_{x}),h|_{E_{x}}}$ over $P(E_x)$ with respect to $h|_{E_x}$. Fix a point $(x,[v])\in P(E)$ and consider the following quantity, computed at $(x,[v])$:
\begin{multline*}
\frac{\Xi^{r}}{r!}\wedge\frac{\pi^*\omega^{n-1}}{(n-1)!} \\=
\frac 1{r!(n-1)!}\sum_{\ell=0}^r\binom{r}{\ell}(-1)^\ell\,\bigl(\omega^{FS}_x([v])\bigr)^{r-\ell}\wedge\bigl(\vartheta^E_h(x,[v])\bigr)^\ell\wedge\pi^*\omega^{n-1}.
\end{multline*}
For degree reasons, only one term survives, namely
$$
\frac{\Xi^{r}}{r!}\wedge\frac{\pi^*\omega^{n-1}}{(n-1)!} =
-\frac{\bigl(\omega^{FS}_x([v])\bigr)^{r-1}}{(r-1)!}\wedge\vartheta^E_h(x,[v])\wedge\frac{\pi^*\omega^{n-1}}{(n-1)!}.
$$
Now, on the one hand
$$
\begin{aligned}
\vartheta^E_h(x,[v])\wedge\frac{\pi^*\omega^{n-1}}{(n-1)!} &= 
\frac i{2\pi}\,\frac {\bigl\langle\Theta_{x}(E,h)\cdot v,v\bigr\rangle_h}{||v||_h^2}\wedge\frac{\pi^*\omega^{n-1}}{(n-1)!} \\
&=
\frac i{2\pi}\,\frac {\bigl\langle\Theta_{x}(E,h)\wedge\frac{\omega^{n-1}}{(n-1)!}\cdot v,v\bigr\rangle_h}{||v||_h^2} \\
&= \frac {\bigl\langle T_{E,h,\omega}(x)\cdot v,v\bigr\rangle_h}{||v||_h^2}\,\frac{\pi^*\omega^n}{n!},
\end{aligned}
$$
where $T_{E,h,\omega}$ is a hermitian endomorphism of $(E,h)$ often called the mean curvature of $(E,h)$ with respect to $\omega$ (observe that by its very definition, $T_{E,h,\omega}\equiv \lambda_E\operatorname{Id}_E$ if and only if $(E,h)$ is Hermite--Einstein with respect to $\omega$). On the other hand, as it is straightforwardly seen again by degree reasons,
$$
\frac{\Xi^{r-1}}{(r-1)!}\wedge\frac{\pi^*\omega^{n}}{n!}=\frac{\bigl(\omega^{FS}_x([v])\bigr)^{r-1}}{(r-1)!}\wedge\frac{\pi^*\omega^{n}}{n!}.
$$
Finally, since $\bigl\langle T_{E,h,\omega}(x)\cdot v,v\bigr\rangle_h$ equals $\lambda_E\,||v||^2_h$ for all $v\in T_{X,x}$ if and only if $T_{E,h,\omega}(x)=\lambda_E\operatorname{Id}_E$, putting all this together, we obtain (\ref{HEPE}).
\end{proof}

\begin{remark}
If $(E,h)$ is not necessarily Hermite--Einstein, then identity (\ref{HEPE}) reads
$$
\frac{\Xi^{r}}{r!}\wedge\frac{\pi^*\omega^{n-1}}{(n-1)!} = -\gamma_1(\vartheta^E_h(x,[v])/\omega)\,\frac{\Xi^{r-1}}{(r-1)!}\wedge\frac{\pi^*\omega^{n}}{n!},
$$
where we define $\gamma_k(\vartheta^E_h(x,[v])/\omega)$, $k=1,\dots,n$, to be the $k$-th symmetric polynomial in the eigenvalues of the real $(1,1)$-form $\vartheta^E_h(x,[v])$ with respect to $\omega$. 
\end{remark}

More generally, the same kind of computations leads, for each $k=1,\dots,n$, to the following identity of top degree forms on $P(E)$.

\begin{proposition}\label{sympoly}
Let $(E,h)\to (X,\omega)$ be a holomorphic hermitian vector bundle of rank $r$ over a $n$-dimensional hermitian manifold. Then, on $P(E)$, we have
\begin{equation}\label{symev}
\frac{\Xi^{r-1+k}}{(r-1+k)!}\wedge\frac{\pi^*\omega^{n-k}}{(n-k)!} = (-1)^k\gamma_k(\vartheta^E_h(x,[v])/\omega)\,\frac{\Xi^{r-1}}{(r-1)!}\wedge\frac{\pi^*\omega^{n}}{n!}.
\end{equation}
\end{proposition}

\begin{proof}
The proof goes exactly as for Proposition \ref{hepe}, with only one supplementary standard remark: for $\alpha$ a real $(1,1)$-form, one has
$$
\frac{\alpha^{k}}{k!}\wedge\frac{\omega^{n-k}}{(n-k)!}=\gamma_k(\alpha/\omega)\,\frac{\omega^n}{n!}.
$$
\end{proof}

We are now in a good shape to prove Theorem \ref{KL}.

\begin{proof}[Proof of Theorem \ref{KL}]
We begin with the following elementary lemma.

\begin{lemma}\label{multlag}
Let $\gamma_1,\gamma_2\colon\mathbb R^n\to\mathbb R$ be respectively the first and second elementary symmetric polynomial function in $n$ variables. Then, $\gamma_2$ has an absolute maximum at the point $(C/n,\dots,C/n)$, if it is subject to the constraint $\gamma_1\equiv C$. 
\end{lemma}

\begin{proof}
Parametrize the affine hyperplane $\gamma_1\equiv C$ by the first $n-1$ variables. Then, a straightforward computation gives
\begin{multline*}
\gamma_2\biggl(x_1+C/n,\dots,x_{n-1}+C/n,C-\sum_{i=1}^{n-1}(x_i+C/n)\biggr)-\gamma_2(C/n,\dots,C/n) \\
= -\sum_{i=1}^{n-1}x_i^2\quad-\sum_{1\le i<j\le n-1}x_ix_j
\\ = -\frac 12\left(\sum_{i=1}^{n-1}x_i\right)^2-\frac 12\sum_{i=1}^{n-1}x_i^2\le 0,
\end{multline*}
and equality holds if and only if $x_i=0$ for all $i=1,\dots,n$.
The absolute maximum is then 
$$
\gamma_2(C/n,\dots,C/n)={n \choose 2}\frac{C^2}{n^2}.
$$
\end{proof}

Now, consider the quantity 
$$
\frac{\Xi^{r+1}}{(r+1)!}\wedge\frac{\pi^*\omega^{n-2}}{(n-2)!} = \gamma_2(\vartheta^E_h(x,[v])/\omega)\,\frac{\Xi^{r-1}}{(r-1)!}\wedge\frac{\pi^*\omega^{n}}{n!}.
$$
Since $(E,h)$ is Hermite--Einstein, we have $\gamma_1(\vartheta^E_h(x,[v])/\omega)\equiv\lambda_E$, so that
$$
\gamma_2(\vartheta^E_h(x,[v])/\omega)\le \binom{n}{2}\,\frac{\lambda_E^2}{n^2}=\frac{n-1}{2n}\lambda_E^2.
$$ 
Therefore
\begin{equation}\label{final}
\frac{\Xi^{r+1}}{(r+1)!}\wedge\frac{\pi^*\omega^{n-2}}{(n-2)!} \le\frac{n-1}{2n}\lambda_E^2\,\frac{\Xi^{r-1}}{(r-1)!}\wedge\frac{\pi^*\omega^{n}}{n!},
\end{equation}
and equality holds if and only if the eigenvalues of $\vartheta^E_h(x,[v])$ with respect to $\omega$ are all equal to $\lambda_E/n$, that is 
$$
\vartheta^E_h(x,[v])=\frac{\lambda_E}{n}\,\omega.
$$
This means that, for all $x\in X$ and $v\in E_x\setminus\{0\}$, we have
$$
\frac i{2\pi}\,\bigl\langle\Theta_{x}(E,h)\cdot v,v\bigr\rangle_h=||v||_h^2\,\frac{\lambda_E}{n}\,\omega,
$$
or, in other words,
$$
\biggl\langle\biggl(\frac i{2\pi}\,\Theta_{x}(E,h)-\frac{\lambda_E}{n}\,\omega\otimes\operatorname{Id}_E\biggr)\cdot v,v\biggr\rangle_h=0, 
$$
that is
$$
\frac i{2\pi}\,\Theta(E,h)=\frac{\lambda_E}{n}\,\omega\otimes\operatorname{Id}_E.
$$
To conclude, we just need to take the push forward of both sides of (\ref{final}) and use Proposition \ref{segre}, to obtain
$$
\frac{s_2(E,h)}{(r+1)!}\wedge\frac{\omega^{n-2}}{(n-2)!}\le\frac{n-1}{2n}\frac{\lambda_E^2}{(r-1)!}\,\frac{\omega^{n}}{n!},
$$
which is precisely (\ref{divin2}).
\end{proof}

To finish with, let us know briefly explain how to formally derive (\ref{divin2}) from the classical Kobayashi--Lübke inequality. Same notations and hypotheses as in the statement of Theorem \ref{KL}, we have
$$
\begin{aligned}
c_1(E,h)^2\wedge\omega^{n-2}-c_2(E,h)\wedge\omega^{n-2} 
& = \frac 1{2r}\bigl((r+1)\,c_1(E,h)^2 \\
&\qquad+(r-1)\,c_1(E,h)^2-2r\,c_2(E,h)\bigr)\wedge\omega^{n-2} \\
& \le \frac{r+1}{2r}\,c_1(E,h)^2\wedge\omega^{n-2},
\end{aligned}
$$
the last inequality being exactly the content of the Kobayashi--Lübke inequality (and so with equality if and only if $E$ is projectively flat). Now, by the primitive decomposition formula, we write
$$
c_1(E,h)=\eta+f\,\omega,
$$
where $\eta$ is a real $\omega$-primitive $(1,1)$-form, \textsl{i.e.} $\Lambda_\omega\eta\equiv 0$, so that $\eta\wedge\omega^{n-1}\equiv 0$, and $f$ is a smooth real function on $X$. Next, from (\ref{lambda}), we obtain that $f$ is constantly equal to $\frac rn\lambda_E$. Then, 
$$
\begin{aligned}
\frac{r+1}{2r}\,c_1(E,h)^2\wedge\omega^{n-2} & = \frac{r+1}{2r}\biggl(\eta^2\wedge\omega^{n-2}+\lambda_E^2\frac {r^2}{n^2}\,\omega^n\biggr) \\
& = \frac{r(r+1)}{2n^2}\lambda_E^2\,\omega^n + \frac{r+1}{2r}\,\eta^2\wedge\omega^{n-2}.
\end{aligned}
$$
It remains to show that $\eta^2\wedge\omega^{n-2}\le 0$, with equality if and only if $c_1(E,h)=\lambda_E\,\frac rn\,\omega$. For this, let
$$
\eta=i\sum_{j=1}^n\alpha_j\,dz_j\wedge d\bar z_j
$$
be a diagonalization of $\eta$ with respect to $\omega$. Then, $\Lambda_\omega\eta=0$ reads $\sum_j\alpha_j=0$ and, as in the proof of Proposition \ref{sympoly}, we get
$$
\eta^2\wedge\omega^{n-2}=\frac{2!(n-2)!}{n!}\,\sum_{j<k}\alpha_j\alpha_k\,\omega^n.
$$
The conclusion follows once again using Lemma \ref{multlag}.

%% file: section3.tex
\section{Final remarks}

Let us finish this note with a few remarks, also in order to underline some of the virtues and shortcomings of the methods presented here.

\begin{remark}
The equality case in (\ref{divin}) gives the projective flatness type condition
$$
\frac{i}{2\pi}\,\Theta(E,h)=\frac{\lambda_E}{n}\,\omega\otimes\operatorname{Id}_E.
$$ 
This easily seen to be actually stronger than the usual projective flatness. Just take any line bundle $L\to X$: of course, as every holomorphic line bundle, $L$ admits a Hermite--Einstein metric $h$ with respect to any $\omega$ and moreover $L$ is trivially projectively flat, but $c_1(L,h)$ cannot equal $\frac{\lambda_L}{n}\,\omega$, unless we had already chosen $\omega$ to be $\frac{n}{\lambda_L}\,c_1(L,h)$ (in particular $i\,\Theta(L,h)$ should have a sign, or be zero). On the other hand, if $(E,h)\to(X,\omega)$ is Hermite--Einstein and projectively flat, then from the classical Kobayashi--Lübke inequality we get
$$
c_2(E,h)\wedge\omega^{n-2}=\frac{r-1}{2r}\,c_1(E,h)^2\wedge\omega^{n-2}.
$$
Now, if we plug the above equality in (\ref{divin}) we obtain
$$
c_1(E,h)^2\wedge\omega^{n-2}\le\lambda_E\,\frac{r}{n}\,c_1(E,h)\wedge\omega^{n-1}=\biggl(\lambda_E\,\frac rn\biggr)^2\,\omega^n.
$$
This last inequality thus can be thought to somehow measure how far $\frac 1r\,c_1(E,h)$ is from being $\frac{\lambda_E}{n}\,\omega$.
\end{remark}

\begin{remark}
If $(E,h)\to (X,\omega)$ is Hermite--Einstein with slope $\lambda_{E,[\omega]}$, then it is also Hermite--Einstein with respect to $t\omega$ for any positive real number $t$ and 
$$
\lambda_{E,[t\omega]}=\frac 1t\,\lambda_{E,[\omega]}.
$$
Thus, for the sake of simplicity, we can normalize $\omega$ and suppose that it has total mass $\int_X \omega^{\dim X}=1$. Suppose $(X,\omega)$ is a smooth compact Kähler surface. In this situation, the classical (integrated) Kobayashi--Lübke inequality reads
$$
c_1(E)^2\le\frac{2r}{r-1}\,c_2(E),
$$
whilst inequality (\ref{divin2}) becomes
$$
c_1(E)^2\le c_2(E)+\lambda_E^2\frac{r(r+1)}8.
$$
If $c_1(E)^2\le 0$ and $c_2(E)\ge 0$, then the two inequalities don't give any further information. On the other hand, if both $c_1(E)^2$ and $c_2(E)$ have the same sign, then the latter is stronger than the former whenever
$$
c_2(E)+\lambda_E^2\frac{r(r+1)}8<\frac{2r}{r-1}\,c_2(E),
$$
that is, as soon as
\begin{equation}\label{stronger}
c_2(E)>\frac{r-1}{2r}\,\bigl(c_1(E)\cdot[\omega]\bigr)^2.
\end{equation}
Thus, when both Chern numbers positive, inequality (\ref{stronger}) provides a non trivial condition which ensures that the Kobayashi--Lübke inequality in its classical incarnation is actually weaker than (\ref{divin2}). This is of course of some usefulness only if one is able to compare \textsl{a priori} the second Chern number of $E$ with its slope.
\end{remark}

\begin{remark}
As pointed out in \cite{Miy91}, according to \cite{Mar77,Mar78}, if we fix the base space $X$, the rank $r$, $c_1$ and $c_2$, the isomorphism classes of stable vector bundles are parametrized by a finite dimensional quasi-projective variety and, in particular, the possibilities of higher Chern classes are finite. Nevertheless, it would be quite useful for Riemann--Roch type computations to find natural inequalities between higher Chern classes of a (semi)stable vector bundle \cite[Problem 4.1]{Miy91}. Unfortunately, the methods presented in this note cannot be straightforwardly adapted to find such inequalities. The reason is that in Lemma \ref{multlag} nothing can be said about the boundedness of higher elementary symmetric polynomials once only the first is supposed to be constant. One could be then led to consider \lq\lq higher order Hermite--Einstein metrics\rq\rq{}, \textsl{i.e.} hermitian metrics $h$ on $E$ such that for some $\ell=1,\dots,n$ the $\gamma_k(\vartheta_h^E/\omega)$'s are constant for $1\le k\le \ell$. Let us say that such a metric $h$ is $\ell$-Hermite--Einstein with respect to $\omega$ (with this definition, thanks to Proposition \ref{hepe}$, (E,h)$ is Hermite--Einstein with respect to $\omega$ if and only if it is $1$-Hermite--Einstein). In our opinion, this may definitely be worth to be investigated, especially in connection with recent developments in the theory of Hessian equations on compact Kähler manifolds (see for instance, juste to cite a few, \cite{Blo05,Hou09,Lu13}).
\end{remark}

\begin{remark}
When $E=T_X$ and $\omega$ is a Kähler--Einstein metric on $X$, then the Guggenheimer--Yau inequality states
$$
\bigl(n\,c_1(X,\omega)^2-(2n+2)\,c_2(X,\omega)\bigr)\wedge \lambda_{T_X}\,c_1(X,\omega)^{n-2}\le 0,
$$
if $\lambda_{T_X}\ne 0$, and $c_2(X,\omega)\wedge\omega^{n-2}\ge 0$, if $\lambda_{T_X}=0$. This stronger Kobayashi--Lübke type inequality relies on the additional symmetries, in the specific case of the tangent bundle, that the curvature tensor acquires whenever computed starting from a Kähler metric. It seems to us that it is not possible to derive such an inequality with our methods, since it is not clear how to take advantage of these further symmetries just looking at the curvature of the tautological bundle on the projectivized bundle of lines of $T_X$.
\end{remark}

\begin{remark}
It is tempting to apply the same kind of techniques on other fiber bundles, such as Grassmannian bundle or, more generally, flag bundles associated to $E$. This would possibly give other inequalities on Chern classes of $E$, as well as ---in the spirit of Guler's theorem mentioned in the introduction--- positivity of more general combinations of Chern classes beside the signed Segre forms, in the case of a Griffiths positive vector bundle. This issue will be addressed in a forthcoming paper.
\end{remark}